\documentclass[letterpaper,11pt,reqno]{amsart}

\makeatletter

\usepackage{amssymb}
\usepackage{latexsym}
\usepackage{amsbsy}
\usepackage{amsfonts}
\usepackage{caption}
\usepackage{subcaption}
\usepackage{hyperref}
\usepackage{graphicx}
\usepackage{enumerate}
\usepackage{enumitem}
\usepackage{color}
\usepackage[round]{natbib} 
\usepackage{comment}
\usepackage{scalerel}
\usepackage{tikz}

\def\marginpar#1{\ignorespaces}

\newcommand{\pr}{{\bf\sf P}}    
\newcommand{\ex}{{\bf\sf E}}              
\newcommand{\var}{{\bf\sf Var}}
\newcommand{\Dir}{\mbox{Dir}}    
\newcommand{\Beta}{\mbox{Beta}}   
\newcommand{\al}{\alpha}    

\def\pol{\texttt{Poly}(\alpha)}
\newtheorem{theorem}{Theorem}
\newtheorem{lemma}[theorem]{Lemma}
\newtheorem{proposition}[theorem]{Proposition}
\newtheorem{corollary}[theorem]{Corollary}

\numberwithin{equation}{section}
\makeatother

\newcommand{\beq}{\begin{equation}}
\newcommand{\eeq}{\end{equation}}

\newcommand{\bal}{\begin{align}}
\newcommand{\eal}{\end{align}}
\newcommand{\bals}{\begin{align*}}
\newcommand{\eals}{\end{align*}}

\begin{document}
\title[Polynomial Voting Rules]{Polynomial Voting Rules}

\author[Wenpin Tang]{{Wenpin} Tang}
\address{Department of Industrial Engineer and Operations Research, Columbia University. 
} \email{wt2319@columbia.edu}

\author[David Yao]{David D.\ Yao}
\address{Department of Industrial Engineer and Operations Research, Columbia University. 
} \email{yao@columbia.edu}

\date{\today} 
\begin{abstract}
We propose and study a new class of polynomial voting rules for a general
decentralized decision/consensus system, and more specifically for
the PoS (Proof of Stake) protocol. 
The main idea, inspired by the Penrose square-root law and the
more recent quadratic voting rule, 
is to differentiate a voter's voting power and the voter's share 
(fraction of the total in the system).
We show that while voter shares form a martingale process that converge to a Dirichlet distribution,
their voting powers follow a super-martingale process that decays to zero over time.
This prevents any voter from controlling the voting process, 
and thus enhances security. 
For both limiting results, we also provide explicit rates of convergence.
When the initial total volume of votes (or stakes) is large, 
we show a phase transition in share stability (or the lack thereof),
corresponding to the voter's initial share relative to the total.
We also study the scenario in which trading (of votes/stakes) among the voters is allowed, 
and quantify the level of risk sensitivity (or risk averse) in three categories, corresponding to the
voter's utility being a super-martingale, a sub-martingale, and a martingale. For each category, 
we identify the voter's best strategy in terms of participation and trading.
\end{abstract}
\maketitle



\setcounter{tocdepth}{1}
\section{Introduction}

\quad Voting, in the traditional sense, refers to a set of rules for a community of individuals or groups
(``voters'') to reach an agreement, or to make a collective decision on some choices and ranking problems.
In today's world, ``voting'' has become a ubiquitous notion that includes any decentralized 
decision-making protocol or system, where the voters are often abstract entities (``virtual''), 
and the voting process automated;    
 and the purpose of reaching consensus is often non-social and non-political, such as 
 to enhance the overall security of an industrial operation or infrastructure (\cite{G82, LSP82}).
Examples include cloud computing (\cite{AF09, DG08}),
smart power grids (\cite{Huang09}),
and more recently, trading or payment platforms and exchanges built 
upon the blockchain technology (\cite{Naka08, Wood14}).

\quad At the core of a blockchain is the {\em consensus protocol}, 
which specifies a set of voting rules for the participants (``miners'' or validators) 
to agree on an ever-growing log of transactions (the ``longest chain'') so as to form
a distributed ledger.
There are several existing blockchain protocols, 
among which the most popular are {Proof of Work} (PoW, \cite{Naka08}) 
and {Proof of Stake} (PoS, \cite{KN12, Wood14}). 
In the PoW protocol, 
miners compete with each other by solving a hashing puzzle. 
The miner who solves the puzzle first receives a reward (a number of coins) and whose work
validates a new block's addition to the blockchain. 
Hence, while the competition is open to everyone,
the chance of winning is proportional to a miner's computing power.

\quad In the PoS protocol, there is a bidding mechanism to select a miner
to do the work of validating a new block.  
Participants who choose to join the bidding are required to commit some ``stakes''
(coins they own), and the winning probability is proportional to the number of stakes committed.
Hence, a participant in a PoS blockchain is actually a ``bidder,'' as opposed to a ``miner'' 
--- only the winning bidder becomes the miner validating the block.
(Any participant who chooses not to join the bidding can be viewed as 
a bidder who commits zero stakes.)
Needless to add, bidding exists long before the PoS protocol, and has been widely used in 
many applications, such as auctions and initial public offerings (IPOs).

\quad Let's explore the PoS bidding mechanism a bit more formally.
Suppose a voter (or bidder) $k$ is in possession of $n_{k,t}$ votes (or stakes) 
at time $t$, an index that counts the rounds of voting or bidding in the protocol;
and $N_t:=\sum_k n_{k,t}$ is the total number of 
votes over all voters. Hence, voter $k$'s share, fraction of the total, is $\pi_{k,t}:=n_{k,t}/N_t$. 
Following a traditional voting rule, voter $k$'s chance or probability of winning, which we call 
{\it voting power}, will be equal to $\pi_{k,t}$, voter $k$'s share. 
Yet, this doesn't have to be the case. That is, 
any voter's voting power needs not be equal to the voter's share (of the system total).
Indeed, there are often good reasons for the two to be different. 

\quad Historically, the English scholar Lionel Penrose famously proposed 
a square-root voting rule (\cite{Pen46}),
around the time when the United Nations was founded shortly after WWII. 
According to Penrose, a world assembly such as the UN should designate 
each country a number of votes that is proportional to the square root of its population. 
The obvious implication (which may or may not be what Penrose initially intended)
is to limit the voting power of nations with very large populations.
In the same spirit, the quadratic voting rule has attracted much attention in recent years
(\cite{LW18}).
The idea is that each voter be given a budget (in dollars, for instance); 
the voter can cast
multiple votes on any single or subset of choices or candidates on the ballot, 
with $x$ votes (for any choice) costing $x^2$ dollars.
Under both voting rules, the voting power is different from the voter's 
share or representation in the system,
population in the first case, and the budget in the second case.

\quad Inspired by these ideas, we propose a class of polynomial voting rules, 
denoted $\pol$, which grant every voter $k$
a voting power that scales the voter's share $\pi_{k,t}$ by a factor $N_t^{-\al}$ for $\al \ge 0$. 
When $\al=0$, this reduces to the traditional voting case of power$=$share, which is a linear rule.
When $\al=1$, the rule resembles the square-root or the quadratic voting rules mentioned above in 
spirit, in terms of decoupling voting power from a voter's share,   
but of course differs in both the application context and implementation schemes.
As we will demonstrate, the general $\pol$ rule is a time change of the $\texttt{Poly}(0)$ rule,
with the parameter $\alpha$ measuring how much the traditional $\alpha =0$ rule is ``slowed down,''
namely, the voting power is diminished over time.

\quad There are (at least) two reasons to consider slowed-down voting schemes in blockchains.
\begin{itemize}[itemsep = 3 pt]
\item
First, the block-generation time requires to be lower bounded due to network delay (see \cite[Section 14.3]{Shibook}).
Specifically, there is the principle of security:
\begin{equation}
\label{security}
(1 - v) \cdot \mbox{honest power} > \gamma \cdot \mbox{dishonest power}, \quad \mbox{or} \quad  v < 1 -  \frac{\gamma \cdot \mbox{dishonest power}}{\mbox{honest power}}.
\end{equation}
Here ``honest/dishonest power'' refers to the voting power of honest/dishonest bidders.  
The parameter $\gamma$ is a user-defined {\em security factor};
e.g., $\gamma = 2$ means, honest power is expected to be twice as much as dishonest power;
hence, $\gamma$ measures how secure a distributed system is.
When honest bidders broadcast their validation results, 
dishonest bidders may exploit network delay to attack;
equivalently, network delay will reduce the honest power.
Thus the term ``$1-v$'' plays the role of a discount factor, with 
$v$ proportional to network delay
(the more severe network delay is, the smaller the discount $1-v$ is, and hence the larger $v$ is).
Honest power is discounted also because honest bidders follow exactly the protocol
while dishonest bidders do not comply with the rules. 
As we will illustrate (in the remarks following Theorem \ref{thm:1})
slowing down the voting process enhances security.
This is because decreasing voting power over time will increase the block generation time,
which will mitigate network delay
and make the principle of security \eqref{security} ``easier" to hold.
\item
Second, PoS blockchains suffer from malicious attacks known as {\em Nothing at Stake} (see e.g. \cite{DPP19}). 
As pointed out in \cite{BD19},  
for the PoS longest-chain protocol, 
honest bidders focus exclusively on the longest chain
while dishonest bidders can work simultaneously on all existing blocks. 
They showed that the PoS longest-chain is less secure than its PoW counterpart,
assuming both honest and dishonest parties have constant voting power over time. 
However, as dishonest bidders have more flexibility, 
it is (much) more likely that they win and get rewarded, 
and their advantage is only amplified over time. 
This makes ``constant voting power" highly undesirable. 
There are two general approaches to solving this problem: 
(i) adjust the amount of reward over time, (ii) slow down the voting process; 
both are aimed at preventing
dishonest bidders from overpowering honest bidders as time evolves.
\end{itemize}

\quad Here is an overview of our main findings and results. 
We prove that under the $\pol$ voting rule,
voter shares form a martingale process that converges to a Dirichlet distribution as $t\to\infty$,  
while their voting powers follow a super-martingale process that decreases to zero 
over time (Theorem \ref{thm:1}); and for both limits we also explicitly characterize their rates of convergence.
Thus, the $\pol$ voting scheme enhances {security},
preventing any voter or any group of voters from controlling the voting process
and overpowering the system.

\quad We further group the voters into two categories: {large} and {small}, according
to the initial (time zero) votes they own relative to the total ($N_0$).
When $N_0$ is large, which is the case in most applications,
we show a phase transition in the stability of voter shares across the two categories
 (Proposition \ref{thm:2}).
Notably, the same phenomenon is demonstrated under the traditional voting rule ($\al=0$),
refer to \cite{RS21,Tang22}.
Our result establishes that this phase transition is in fact {\it universal}, 
in the sense that it applies to all values of $\al (\ge 0)$. 

\quad We also study the scenario in which trading (of votes/stakes) among the voters (or ``bidders'') is allowed, 
motivated by PoS applications in cryptocurrency. 
For $\al=0$, the trading scenario has been recently studied in \cite{RS21}.
Not only our model is more general, in allowing any $\al\ge 0$, 
our results are also richer and sharper  (Theorem \ref{thm:3}).
For instance, we quantify the level of risk sensitivity (or risk-averse) that results in three cases
according to the voter's utility being a super-martingale, a sub-martingale or a martingale.  
Each case will lead to a best strategy for the voter, including ``non-participation'' (not to participate at all in the bidding) 
and ``buy out'' (buying as many stakes as what is available), which are not considered in \cite{RS21}.
Note that ``buy-out" is a monopoly, 
and it is desirable to limit the number of stakes that any voter can acquire in a single round.
This is studied in our subsequent paper \cite{TY23} on trading PoS stakes with volume constraint.
See also \cite{Tang23} for various problems (including transaction costs and voter's collective behavior)
related to the PoS trading.


\quad The key to our analysis relies on the study of the random process $N_t$ 
(the total volume of votes/stakes at time $t$), 
which is a time-homogeneous Markov chain.
We develop some asymptotic results for this Markov chain, including large-deviation bounds 
(Theorem \ref{prop:2})
and a fluid limit
(Proposition \ref{prop:1}).

\medskip

\quad In the remainder of this paper there are two main sections.
 Section \ref{sc2} studies the $\pol$ voting model from a general perspective,
 focusing on the associated stochastic processes, such as $N_t$, 
 voter shares, and voting powers; their long-term behavior and limits, 
 some of which are further characterized by concentration inequalities or large-deviation bounds.
 Section \ref{sc3} concerns two aspects of the $\pol$ voting rule that are more 
 closely associated with the application of PoS in cryptocurrency: 
 (a) the evolution of bidder shares over time and
 the phase transition phenomenon mentioned above;
  (b) the issue of incentive and risk-sensitivity when trading is allowed.
Concluding remarks and suggestions for further research are collected 
in Section \ref{sc5}.

\section{The $\texttt{Poly}(\alpha)$ Voting Model}
\label{sc2}

\quad In this section, we develop a formal model 
for the $\pol$ voting rule,
focusing on the stochastic processes associated with the model, 
their properties and limiting behavior.

\quad First, here is a list of some of the common notation used throughout the paper.
\begin{itemize}[itemsep = 3 pt]
\item
$\mathbb{N}_{+}$ denotes the set of positive integers, and $\mathbb{R}$ denotes the set of real numbers.
\item
$\stackrel{d}{=}$ denotes equal in distribution, and $\stackrel{d}{\longrightarrow}$ denotes convergence in distribution.
\item
$a = \mathcal{O}(b)$ means $\frac{a}{b}$ is bounded from above as $b \to \infty$;
$a = \Theta(b)$ means $\frac{a}{b}$ is bounded from below and above as $b \to \infty$;
and 
$a = o(b)$ or $b \gg a$ means $\frac{a}{b}$ decays towards zero as $b \to \infty$.
\item
$d_W(\mu, \nu)$ denotes the $1$-Wasserstein distance between two probability distributions $\mu$ and $\nu$. Refer to \cite[Chapter 6]{Villani09}.
\end{itemize}
We use $C, C' , C''$ etc to denote generic constants (which may change from line to line).

\quad The voters, referred to as {\it bidders} below, are the participants in the decentralized system, 
where they engage in rounds of bidding following a pre-specified voting rule (the ``consensus
protocol'') so as to win more votes, or {\it stakes}. 
(The PoS protocol described in the Introduction provides a concrete instance to motivate the
model here.)
Let $K \in \mathbb{N}_{+}$ be the total number of bidders, which will stay fixed throughout the paper;    
and let $[K]: = \{1, \ldots, K\}$ denote the set of all bidders.

\quad Time is discrete, indexed by $t=0,1,2, \dots$, and corresponds to the rounds of bidding
mentioned above.
Bidder $k$ initially owns  $n_{k,0}$ stakes.
Let $N:=\sum_{k = 1}^K n_{k,0}$ denote the total number of initial stakes owned by all $K$ bidders.
The term {\em bidder share} refers to the fraction of stakes each bidder owns.
So the initial bidder shares $(\pi_{k, 0}, \, k \in [K])$ are given by
\begin{equation}
\label{eq:share0}
\pi_{k, 0}: = \frac{n_{k,0}}{N}, \quad k \in [K].
\end{equation}
Similarly, $n_{k,t}$ denotes the number of stakes owned by bidder $k$ at time $t \in \mathbb{N}_{+}$, 
and the corresponding share is
\begin{equation}
\label{eq:sharet}
\pi_{k,t}:= \frac{n_{k,t}}{N_t}, \quad k \in [K], \quad \mbox{with } N_t:= \sum_{k=1}^K n_{k,t}.
\end{equation}
Here $N_t$ is the total number of stakes at time $t$,  and thus $N_0 = N$.
(We shall often refer to $N_t$ as the ``volume of stakes'' or, simply, ``volume.'')
Clearly, for each $t \ge 0$, $(\pi_{k, t},\, k \in [K])$ forms a probability distribution on $[K]$.

\quad In each period $t$, 
a single stake (or ``reward'') is distributed as follows:
each bidder $k$ receives the reward with probability 
\begin{equation}
\label{eq:powert}
\theta_{k,t} := \frac{n_{k,t}}{N_t^{1+\alpha}}= \frac{\pi_{k,t}}{N_t^{\alpha}}  , 
\end{equation}
and receives nothing with probability $1 - \theta_{k,t}$.
Clearly, $\theta_{k,t}$ is bidder $k$'s {\em reward rate}, 
as $1/\theta_{k,t}$ is the average number of rounds for bidder $k$ to win an additional unit of stake. 
To the extent the reward is coupled with the voting mechanism outlined above, 
$\theta_{k,t}$ can also be viewed as bidder $k$'s voting power at time $t$.
(Below, we use the terms ``reward rate'' and ``voting power'' interchangeably if there is no ambiguity.)
When $\alpha =0$, 
the voting power $\theta_{k,t}$ coincides with the bidder share $\pi_{k,t}$,
which is the P\'olya urn framework in \cite{RS21},  and in \cite{Tang22}.

\quad Let $S_{k,t}$ be the random event that bidder $k$ receives one unit of reward in period $t$.
Thus, the number of stakes owned by each bidder evolves as follows,
\begin{equation}
\label{eq:TP1}
n_{k,t} = n_{k, t-1} + 1_{S_{k,t}}, \quad k \in [K] ;
\end{equation}
or simply,
\begin{equation}
\label{eq:TP2}
n_{k,t} = \left\{ \begin{array}{lcl}
n_{k,t-1} & \mbox{with probability}
& 1 - \theta_{k,t-1}, \\
n_{k,t-1} + 1 & \mbox{with probability} & \theta_{k,t-1}.
\end{array}\right.
\end{equation}
Accordingly, the total number of stakes $N_t$ evolves as follows, taking into account  
$\sum_{k = 1}^K n_{k,t} = N_t$,
\begin{equation}
\label{eq:Nt}
N_t = \left\{ \begin{array}{lcl}
N_{t-1} & \mbox{with probability}
& 1 - 1/N_{t-1}^{\al}, \\
N_{t-1} + 1 & \mbox{with probability} & 1/N_{t-1}^{\al}.
\end{array}\right.
\end{equation}
The counting process $(N_t, \, t \ge 0)$ specified in \eqref{eq:Nt}  evolves as a time-homogeneous Markov chain on $\{N, N+1, \ldots\}$,
in contrast with the P\'olya urn (with a constant reward) in which $N_t$ grows deterministically and linearly in $t$.
As we will see in the next subsection, 
the $\pol$ voting rule slows down the distribution of rewards,
so the volume of stakes grows sublinearly. 
This is consistent with the volume growth in many cryptocurrencies such as Bitcoin and Ethereum.

\quad Let $\pmb{n}_t = (n_{1,t}, \ldots, n_{K,t})$ be the vector of bidder stakes at time $t$. 
An alternative (and useful) characterization of $(\pmb{n}_t, \, t \ge 0)$ is given as follows. 
\begin{proposition}
\label{prop:embed}
Let $(L_t, \, t \ge 0)$ be a counting process with arrivals occurring at $0 = T_0 < T_1 < \cdots$, such that
the inter-arrival times are independent, with $T_{k+1} - T_k$, for every $k\ge 0$, following a geometric distribution with success probability parameter $(N+k)^{-\alpha}$.
Define the process $(\pmb{l}_t, \, t \ge 0)$ by
\begin{equation*}
\pmb{l}_t = \pmb{l}_{T_k} \quad \mbox{for } T_k \le t < T_{k+1},
\end{equation*}
where $(\pmb{l}_{T_k}, \, k \ge 0)$ is a copy of the P\'olya urn process with $K$ colors and $N$ initial balls.
Then, we have $(\pmb{n}_t, \, t \ge 0) \stackrel{d}{=} (\pmb{l}_t, \, t \ge 0)$, where $\pmb{n}_t$ is the process of bidder stakes defined above.
\end{proposition}
\begin{proof}
It is clear from the dynamics in \eqref{eq:Nt} that the two counting processes $(N_t, \, t \ge 0)$ and $(L_t, \, t \ge 0)$ have the same distribution. 
Given $\pmb{n}_t$, 
the probability that the next (unit of) stake goes to bidder $k$ is $\frac{n_{k,t}}{N_t^{1+\alpha}} / \frac{1}{N_t^{\alpha}} = \frac{n_{k,t}}{N_t}$ by the craps principle.
The connection to  the P\'olya urn process with $K$ colors (voters) and $N$ initial balls (stakes) is obvious.
\end{proof}

\quad The above proposition implies that the P\'olya urn is embedded in the process of stakes $(\pmb{n}_t, \, t \ge 0)$ 
through a random time change $(N_t, \, t \ge 0)$. This fact will be used below to study the long-time behavior of 
bidder shares and reward rates in \S\ref{sc22}. 
But we first study in the next subsection how the issuance of rewards is slowed down under the $\texttt{Poly}(\alpha)$ voting rule.

\subsection{The volume $(N_t, \, t \ge 0)$}

Let $\mathcal{F}_t$ be the filtration generated by the random events $(S_{k, r}: k \in [K], \, r \le t)$.

\begin{proposition}[Long-time behavior of $N_t$]
\label{prop:Nt}
Under the $\texttt{Poly}(\alpha)$ voting rule, the following results hold:
\begin{itemize}
\item[(i)]
The process $(N_t, \, t \ge 0)$ is an $\mathcal{F}_t$-sub-martingale, and its compensator is 
\begin{equation*}
A_t = \sum_{k \le t - 1} N_k^{-\alpha} \quad \mbox{for } t \ge 1.
\end{equation*}
\item[(ii)]
There is the convergence in probability:
\begin{equation}
\label{eq:cvpN}
\frac{N_t^{1 + \al}}{t} \longrightarrow
1 + \alpha \quad \mbox{as } t \to \infty.
\end{equation}
\end{itemize}
\end{proposition}

\begin{proof}
(i) It suffices to note that $\ex (N_{t+1} \,|\, \mathcal{F}_t) = N_{t} + N_t^{-\alpha}$, for all $t \ge 0$.

(ii) Apply the method of moments by computing $\ex(N_t^{(1+\al)j})$ for all $j$. 
For $j = 1$, we have by definition
\begin{align*}
\ex (N_{t+1}^{1 + \al} - N_t^{1 + \al} \,|\, N_t = x) & =  (1+x)^{1 + \al} \frac{1}{x^\al} + x^{1 + \al} \left(1 - \frac{1}{x^\al} \right) -x^{1 + \al} \\
& = 1 + \al + \mathcal{O}(x^{-1}) \quad \mbox{as } x \to \infty.
\end{align*}
It is clear that with probability one $N_t \to \infty$ as $t \to \infty$.
As a result, $\ex (N_{t+1}^{1 + \al} - N_t^{1 + \al}) \to 1+ \al$ as $t \to \infty$ which yields
\begin{equation}
\label{eq:moment1}
\ex N_t^{1+\al} \sim (1+ \al) t \quad \mbox{as } t \to \infty.
\end{equation}
Next for $j = 2$, we have
\begin{equation*}
\ex (N_{t+1}^{2(1+\al)} - N_t^{2(1+\al)} \,|\, N_t = x) = 2(1+\al)x^{1+\al} + \mathcal{O}(x^\al) \quad \mbox{as } x \to \infty.
\end{equation*}
Thus, 
$\ex (N_{t+1}^{2(1+\al)} - N_t^{2(1+\al)}) = \left(2(1+ \al) + o(1)\right) \ex N_t^{1+\al} \sim 2(1+\al)^2 t$ by \eqref{eq:moment1}.
Then we get $\ex(N_t^{2 (1+\al)}) \sim (1+\al)^2 t^2$ as $t \to \infty$.
We proceed by induction.
Assuming that $\ex(N_t^{j(1+\alpha)}) \sim (1+\alpha)^j t^j$ as $t \to \infty$,
we get
\begin{equation*}
\ex(N_t^{(1+\alpha)(j+1)} - N_t^{(1+\alpha)(j+1)}) = ((j+1)(1+\alpha) + o(1)) \ex(N_t^{j(1+\alpha)})
\sim (j+1) (1+\alpha)^{j+1} t^j,
\end{equation*}
which implies that $\ex(N_t^{(1+\alpha)(j+1)}) \sim (1+\alpha)^{j+1} t^{j+1}$ as $t \to \infty$.
Thus, we have
\begin{equation*}
\ex(N_t^{(1+\al)j}) \sim (1+\al)^j t^j \quad \mbox{as } t \to \infty, \quad j = 1, 2, \ldots
\end{equation*}
By the method of moments (see e.g. \cite[Section 30]{Bill95}), 
$N_t^{(1+\al)}/t$ converges in distribution, and thus in probability to $1+\al$.
\end{proof}

\quad The proposition gives the growth rate of the volume of stakes: 
$N_t$ grows as $((\al+1) t)^{\frac{1}{1+\al}}$ as $t \to \infty$. 
Part (i) suggests that $N_t \sim \sum_{k \le t-1} N_k^{-\alpha}$, which is consistent with the limit in \eqref{eq:cvpN}.
When $\alpha =0$, $N_t$ follows 
the (deterministic) linear growth of the P\'olya urn model (with a constant reward).
For $\al = 1$, $N_t$ grows as $\sqrt{t}$.

\quad Even more important is the question, how does 
$N_t$ ``fluctuate'' around its growth trajectory $((\al+1) t)^{\frac{1}{1+\al}}$;
specifically, how to establish large-deviation bounds on $N_t$?
This is addressed in the next theorem, along with a corollary that confirms $N_t$'s
concentration around its growth trajectory, for large $t$.

\begin{theorem}[Large deviations for $N_t$]
\label{prop:2}
Define a function $f_\al(\cdot)$,
$$f_\al: \lambda\in (0,\infty) \;\mapsto \;  (1 + \al)\lambda \log \lambda - (1+\al)\lambda + \frac{1}{\lambda^\al} \in \mathbb{R}.$$
Let $\lambda_{-}(\al) < \lambda_{+}(\al)$ be the two roots of $f_\al(\cdot)$ on $(-\infty, \infty)$.
Under the $\texttt{Poly}(\alpha)$  voting rule, the following results hold:
\begin{itemize}
\item[(i)]
For each $\lambda < \lambda_{-}(\al)$, and for any $\varepsilon >0$,
\begin{equation}
\label{eq:lowerdb}
\pr(N_t < \lambda t^{\frac{1}{1+\al}}) \le \exp \left( - (1 - \varepsilon) f_\al(\lambda) \, t^{\frac{1}{1+\al}}\right) \quad \mbox{as } t \to \infty.
\end{equation}
\item[(ii)]
For each $\lambda > \lambda_{+}(\al)$, and for any $\varepsilon >0$,
\begin{equation}
\label{eq:upperdb}
\pr(N_t > \lambda t^{\frac{1}{1+\al}}) \le \exp \left( - (1 - \varepsilon) f_\al(\lambda) \, t^{\frac{1}{1+\al}}\right) \quad \mbox{as } t \to \infty.
\end{equation}
\end{itemize}
\end{theorem}

\begin{proof}
Without loss of generality, assume that $N_0= 1$.
Note that $(N_t, \,  t \ge 0)$ has increments $\{0,1\}$,
so there are $\binom{t}{k}$ paths ending at $N_t = k+1$ 
(one has to choose $k$ upward steps ``$1$'' out of $t$ steps).
Moreover, the probability of each path ending at $N_t = k+1$ is upper bounded by
\begin{equation*}
\frac{1}{(k !)^{\al}}\left(1 - \frac{1}{(k+1)^\al} \right)^{t - k},
\end{equation*}
since the $k$ upward ``$1$'' steps contribute $1/k!$, 
and the remaining $t -k$ flat ``$0$'' steps have at most probability $\left(1 - \frac{1}{(k+1)^\al} \right)^{t - k}$.
Thus,
\begin{equation*}
\pr(N_t \le m+1) \le \sum_{k \le m} a_{k} \quad \mbox{and} \quad \pr(N_t > m) \le  \sum_{k \ge m} a_k
\end{equation*}
where
\begin{equation}
\label{eq:akk}
a_k:= \binom{t}{k} \frac{1}{(k !)^{\al}}\left(1 - \frac{1}{(k+1)^\al} \right)^{t - k}.
\end{equation}
Standard analysis shows that there are $0 < k_1 < k_2$ such that $a_k$ is nondecreasing on $[1, k_1)$ and $[k_2, t)$.
As we will see, $k_1 \sim \lambda_{-}(\al)\, t^{\frac{1}{1 + \al}}$ and $k_2 \sim \lambda_{+}(\al) \, t^{\frac{1}{1+\al}}$ as $t \to \infty$.
The idea is to study the term $a_k$ with $k={\lambda t^{\frac{1}{1+\al}}}$ for $\lambda > 0$ as $t \to \infty$.
By Stirling's formula, 
\begin{equation*}
\binom{t}{\lambda t^{\frac{1}{1+\al}}} \sim \frac{1}{\sqrt{2 \pi \lambda}} t^{-\frac{1}{2(1+\al)}} \exp \left( \frac{\lambda \al}{1+\al} t^{\frac{1}{1 + \al}} \log t + (\lambda - \lambda \log \lambda) t^{\frac{1}{1+\al}} + o\left( t^{\frac{1}{1+\al}}\right) \right),
\end{equation*}
\begin{equation*}
\left(\lambda t^{\frac{1}{1+\al}}\right)! \sim
\sqrt{2 \pi \lambda} \, t^{\frac{1}{2(1+\al)}} \exp \left( \frac{\lambda}{1+\al} t^{\frac{1}{1+\al}} \log t + (\lambda \log \lambda - \lambda)t^{\frac{1}{1+\al}}\right),
\end{equation*}
and
\begin{equation*}
\left(1 - \frac{1}{\lambda^\al t^{\frac{\al}{1+\al}}} \right)^{t-\lambda t^{\frac{1}{1+\al}}} = \exp \left( - \frac{t^{\frac{1}{1+\al}}}{\lambda^\al}
+ o\left(t^{\frac{1}{1+\al}} \right)\right).
\end{equation*}
Combining the above estimates yields
\begin{equation}
a_{\lambda \sqrt{t}} \sim (2 \pi \lambda)^{-\frac{1+\al}{2}} t^{-\frac{1}{2}} 
\exp \left( - f_\al(\lambda) \, t^{\frac{1}{1+\al}}\right).
\end{equation}
Note that 
$f'_\alpha(\lambda) = (1+\alpha)\log \lambda - \alpha \lambda^{-1-\alpha}$, which is increasing from $-\infty$ to $\infty$ on $[0,\infty)$.
The unique stationary point of $f_\alpha$ on $[0,\infty)$ is achieved at $\lambda_*$ such that $\lambda_*^{\alpha + 1} \log \lambda_* = \frac{\alpha}{1+\alpha}$, so it is clear that $\lambda_* > 1$. We have 
\begin{equation*}
f_\alpha(\lambda_*) = (\alpha - 1) \lambda_*^{-\alpha} - (1+\alpha) \lambda_* < 0.
\end{equation*}
Thus, the function $\lambda \to f_\al(\lambda)$ has two roots $\lambda_{-}(\al) < \lambda_{+}(\al)$ on $[0, \infty)$,
and $f_\al > 0$ on $(0, \lambda_{-}(\al)) \cup (\lambda_{+}(\al), \infty)$.
As a result, 
for each $\lambda < \lambda_{-}(\al)$ we have
\begin{equation*}
\pr(N_t < \lambda t^{\frac{1}{1+\al}}) \le C_\lambda t^{-\frac{1}{2} + \frac{1}{1+\al}} \exp \left( - f_\al(\lambda) \, t^{\frac{1}{1+\al}}\right) \le \exp \left( - (1 - \varepsilon) f_\al(\lambda) \, t^{\frac{1}{1+\al}}\right) \, \mbox{as } t \to \infty,
\end{equation*}
and each $\lambda > \lambda_{+}(\al)$ we have
\begin{equation*}
\pr(N_t > \lambda t^{\frac{1}{1+\al}}) \le C'_\lambda t^{\frac{1}{2}} \exp \left( - f_\al(\lambda) \, t^{\frac{1}{1+\al}}\right) \le \exp \left( - (1 - \varepsilon) f_\al(\lambda) \, t^{\frac{1}{1+\al}}\right) \, \mbox{as } t \to \infty.
\end{equation*}
\end{proof}

\quad The theorem above gives exponential deviation bounds for $N_t$ 
when it is either sufficiently small (below $\lambda_{-}(\al) t^{\frac{1}{1+\al}}$),
or sufficiently large (above $\lambda_{+}(\al) t^{\frac{1}{1+\al}}$).
Note that there is a gap between the two bounding curves,
since for each $\al > 0$, 
$\lambda_{-}(\al) < (1+\al)^{\frac{1}{1+\al}} < \lambda_{+}(\al)$.
(For instance, for $\al = 1$,  
$\lambda_{-}(1) \approx 0.56 < \sqrt{2} < 2.51 \approx \lambda_{+}(1)$.)
The gap is due to the combinatorial estimates in our proof, 
which may very well be improved.
Refer to the Appendix \ref{sc41} for a numerical procedure 
that shrinks the gap.

\quad As a corollary, the volume of stakes $N_t$ concentrates around $((1+\al)t)^{\frac{1}{1+\al}}$ for large $t$.

\begin{corollary}
\label{coro:1}
Under the $\texttt{Poly}(\al)$ voting rule, we have, for each $\delta > 0$,
\begin{equation}
\label{eq:concentno}
\pr\left(|N_t - ((1+\al)t)^{\frac{1}{1+\al}} |> \delta t^{\frac{1}{1+\al}}\right) = \mathcal{O}( t^{-\frac{1}{1+\alpha}}) \quad \mbox{as } t \to \infty.
\end{equation}
\end{corollary}

\begin{proof}
Note that $N_0 \le N_t \le t + N_0$, and 
by Theorem \ref{prop:2}, we get $\lambda_1 t^{\frac{1}{1+\alpha}} \le N_t \le \lambda_2 t^{\frac{1}{1+\alpha}}$
with probability $1 - \exp(-C \lambda t^{\frac{1}{1+\alpha}})$ for some $\lambda_1, \lambda_2, C > 0$.
Thus,
\begin{equation*}
\mathbb{E}N_t^{-1} = \mathcal{O}(t^{-\frac{1}{1+\alpha}}) \quad \mbox{and} \quad
\mathbb{E}N_t^\alpha = \mathcal{O}(t^{\frac{\alpha}{1+\alpha}}).
\end{equation*}
According to the proof of Proposition \ref{prop:Nt} (ii), we have
\begin{equation*}
\mathbb{E}(N_{t+1}^{1+\alpha} - N_t^{1 + \alpha}) = 1+ \alpha + \mathcal{O}(\mathbb{E}N_t^{-1})
\quad \mbox{and} \quad
\mathbb{E}(N_{t+1}^{2(1+\alpha)} - N_t^{2(1 + \alpha)}) = 2(1+\alpha) \mathbb{E}N_t^{1+\alpha} + \mathcal{O}(\mathbb{E}N_t^{\alpha}).
\end{equation*}
Therefore,
$\mathbb{E}N_t^{1+\alpha} = (1+\alpha)t + \mathcal{O}(t^{\frac{\alpha}{1+\alpha}})$ and 
$\mathbb{E}(N_t^{2(1+\alpha)}) = (1+\alpha)^2 t^2 + \mathcal{O}(t^{\frac{1+2 \alpha}{1+\alpha}})$,
which implies $\var(N_t^{1+\alpha}) = \mathcal{O}(t^{\frac{1+2 \alpha}{1+\alpha}})$.
Hence,
\begin{equation*}
\pr\left(|N_t^{1+\alpha} - (1+\alpha) t| > \delta t \right) = \mathcal{O}(t^{-\frac{1}{1+\alpha}}) \quad \mbox{for } t \to \infty.
\end{equation*}
Taking $\lambda < \lambda_-(\alpha)$, we have
\begin{align*}
& \quad \, \pr\left(|N_t - ((1+\al)t)^{\frac{1}{1+\al}} |> \delta t^{\frac{1}{1+\al}}\right) \\
& \le \pr\left(N_t < \lambda t^{\frac{1}{1+\alpha}} \right) 
+  \pr\left(|N_t - ((1+\al)t)^{\frac{1}{1+\al}} |> \delta t^{\frac{1}{1+\al}}, \, N_t \ge \lambda t^{\frac{1}{1+\alpha}} \right) \\
& \le \exp(- C' t^{\frac{1}{1+\alpha}}) + \pr(|N_t^{1+\alpha} - (1+\alpha) t| > C''t),
\end{align*}
for some $C', C'' > 0$ (depending on $\alpha, \delta, \lambda$).
Combining the above estimates yield \eqref{eq:concentno}.
\end{proof}

\quad Recall that the process $(N_t, \, t \ge 0)$ is a time-homogenous Markov chain.
The path properties of a general time-homogenous Markov chain $(Z_t, \, t \ge 0)$ has long been studied 
since the work of \cite{Lam60, Lam62, Lam63}.
The basic idea is to study the recurrence, or transience of $(Z_t, \, t \ge 0)$ based on
\begin{equation*}
m_1(x) = \ex(Z_{t+1} - Z_t \,|\, Z_t = x) \quad \mbox{and} \quad m_2(x) = \ex((Z_{t+1} - Z_t)^2 \,|\, Z_t = x).
\end{equation*}
For instance, if $\limsup_{x \to \infty} 2x m_1(x) + m_2(x) \le 0$ then $(Z_t, \, t \ge 0)$ is recurrent;
and if $\liminf_{x \to \infty} 2x m_1(x) + m_2(x) > 0$ then $(Z_t, \, t \ge 0)$ is transient.
The regime corresponding to $m_1(x) = o(1)$ is called the {\em Markov chain with asymptotic zero drift},
and features active research (see e.g. \cite{DKW16, MPW17}).
Specializing to the process $(N_t, \, t \ge 0)$, we have
\begin{equation*}
m_1(x) = m_2(x) = \frac{1}{x^\al}.
\end{equation*}
Interestingly, there seem to be few results on the Lamperti's problem where both $m_1(x)$ and $m_2(x)$ decreases to zero as $x \to \infty$, except that $(N_t, \, t \ge 0)$ is transient.
Observe that $(N_t, \, t \ge 0)$ is nondecreasing, and 
\begin{equation*}
\var(N_{t+1} \,|\, N_t =x) = \left(1 - \frac{1}{x^\alpha} \right)^2 \frac{1}{x^\alpha} + \left(-\frac{1}{x^\alpha}\right)^2 \left(1 - \frac{1}{x^\alpha} \right),
\end{equation*}
where the ``upward" contribution $\left(1 - \frac{1}{x^\alpha} \right)^2 \frac{1}{x^\alpha}$
is larger than the ``downward'' counterpart $\frac{1}{x^{2 \alpha}} \left(1 - \frac{1}{x^\alpha} \right)$
as $x \to \infty$.
In the similar spirit to \cite{Lam62}, the asymptotic growth \eqref{eq:cvpN} hinges on a degenerate fluid approximation of the process $(N_t, \, t \ge 0)$, as stated in the following proposition.

\begin{proposition}[Fluid limit of $N_t$]
\label{prop:1}
Under the $\texttt{Poly}(\alpha)$ voting rule, we have
\begin{equation}
\label{eq:fluid}
\left( \frac{N_{nu}}{n^{\frac{1}{1+\al}}}, \, u \ge 0 \right) \stackrel{d}{\longrightarrow} (X_u, \, u \ge 0) \quad \mbox{as } n \to \infty \mbox{ in } \mathcal{C}[0,\infty),
\end{equation}
where $N_s$ for non-integer $s$ is defined by the linear interpolation of the chain $(N_t, \, t \ge 0)$,
and $X_u = \left( (1 + \al) u\right)^{\frac{1}{1+\al}}$, $u \ge 0$ is the solution to the 
ordinary differential equation $dX_u = X_u^{-\al} dt$ with $X_0 = 0$.
\end{proposition}

\begin{proof}
Fix $T > 0$. 
It suffices to prove the weak convergence $\eqref{eq:fluid}$ on $[0,T]$.
By Proposition \ref{prop:Nt} (ii), there is the convergence in probability
$N_{[nT]}/n^{\frac{1}{1+\al}} \to X_T$ as $n \to \infty$.
Given $\varepsilon > 0$, there is $n(\varepsilon) > 0$ such that for any $n > n(\varepsilon)$,
\begin{equation*}
\pr \left(N_{[nT]}/n^{\frac{1}{1+\al}}  < 2X_T \right) > 1 - \varepsilon.
\end{equation*}
Let $K(\varepsilon): = \max(2X_T, \max_{n \le n(\varepsilon)} (N_0 + [nT])/n^{\frac{1}{1+\al}})$.
We have 
$\pr(N_{[nT]}/n^{\frac{1}{1+\al}} < K(\varepsilon)) > 1 - \varepsilon$ for each $n \in \mathbb{N}_{+}$.
Note that for each $n \in \mathbb{N}_+$, the process $N^{n, T}:= \left(N_{[nT]}/n^{\frac{1}{1+\al}}, \, 0 \le t \le T \right)$ is nondecreasing.
Thus, 
\begin{equation*}
\pr \left(N^{n,T} \in [0,T] \times [0, K(\varepsilon) ] \right) > 1 - \varepsilon.
\end{equation*}
So the sequence of processes $(N^{n,T}, \, n \in \mathbb{N}_{+})$ is tight.
Moreover, for each $t \in [0,T]$, $N_{[nt]}/n^{\frac{1}{1+\al}}$ converges in probability to $X_t$ as $n \to \infty$.
Then for $0 \le t_1 < \cdots < t_k$, 
the vector $(N_{[nt_1]}/n^{\frac{1}{1+\al}}, \cdots, N_{[nt_k]}/n^{\frac{1}{1+\al}})$ converges in probability to $(X_{t_1}, \cdots, X_{t_k})$, i.e. the convergence in finite-dimensional distributions.
The weak convergence follows readily from the tightness and the convergence in finite-dimensional distributions 
(see e.g. \cite[Chapter 2]{Bill99}).
\end{proof}

\quad Note that the fluid limit in the proposition is different from the fluid limit in the literature of stochastic networks, where it usually takes the form of a FSLLN (functional strong law of large numbers) concerning a renewal process and the associated counting process.
 In that setting, the convergence (to a deterministic function of time) is stronger -- the almost sure convergence, and uniformly on $[0,T]$. 
 Refer to \cite[\S6.1]{CY01}.
Here, the process $(N_t, \, t \ge 0)$ is non-renewal, thus the FSLLN limit does not apply;
yet there is still the {\it weak} convergence, and the limit is still a deterministic function of time $(X_u, \, u \ge 0)$, explicitly characterized above. 
Another notable point is, in the FSLLN setting, both time and space are scaled by the same scaling factor $n$
whereas in \eqref{eq:fluid} the time scaling remains the same, and the space scaling is by $n^{\frac{1}{1+\al}}$.
But this is only because $(N_t, \, t \ge 0)$ grows in the order of $t^{\frac{1}{1+\al}}$ (Proposition \ref{prop:Nt} (ii)), 
whereas a renewal (counting) process grows linearly in $t$. 

\subsection{Bidder shares and voting powers}
\label{sc22}
Here we study the evolution and the long-time behavior of $(\pi_{k,t}, \, k \in [K])$ and $(\theta_{k, t},\, k \in [K])$.
Recall that the Dirichlet distribution with parameters $(a_1, . . . , a_K)$, which we denote by $\Dir(a_1, . . . , a_K)$, 
has support on the standard simplex $\{(x_1, \ldots, x_K) \in \mathbb{R}_+^K: \sum_{k = 1}^K x_k = 1\}$ and has density 
\begin{equation*}
f(x_1, \ldots x_K) = \frac{\Gamma(\sum_{k = 1}^K a_k)}{\prod_{k = 1}^K \Gamma(a_k)} \prod_{k = 1}^K x_k^{a_k - 1},
\end{equation*}
where $\Gamma(z) = \int_0^\infty x^{z-1} e^{-x} dx$ is the Gamma function.
For $K = 2$, the Dirichlet distribution reduces to the beta distribution, denoted as $\Beta(a_1, a_2)$.
It is easily seen that if $(x_1, \ldots, x_K) \stackrel{d}{=} \Dir(a_1, \ldots, a_K)$ 
then for each $k \in [K]$,  $x_k \stackrel{d}{=} \Beta(a_k, \sum_{j \ne k} a_j)$.

\begin{theorem}[Long-time behavior]
\label{thm:1}
Under the $\texttt{Poly}(\alpha)$ voting rule, we have the following limiting distributions.
\begin{itemize}
\item[(i)]
Bidder shares: the process $(\pi_{k,t}, \, t \ge 0)$ is an $\mathcal{F}_t$-martingale, and with probability one,
\begin{equation}
\label{eq:mglecv}
(\pi_{1,t}, \ldots, \pi_{K,t}) \longrightarrow (\pi_{1,\infty}, \ldots, \pi_{K,\infty}) \quad \mbox{as } t \rightarrow \infty,
\end{equation}
where $(\pi_{1,\infty}, \ldots, \pi_{K,\infty}) \stackrel{d}{=} \Dir(n_{1,0}, \ldots, n_{K,0})$. 
Moreover, for each $k \in [K]$,
\begin{equation}
\label{eq:Warate}
d_W\left(\pi_{k,t},  \, \Beta\left(n_{k,0}, \, N - n_{k,0}\right)\right) = \mathcal{O}(t^{-\frac{1}{1+\alpha}}) \quad \mbox{as } t \to \infty.
\end{equation}
\item[(ii)]
Voting powers: the process $(\theta_{k,t}, \, t \ge 0)$ is an $\mathcal{F}_t$-super-martingale, and
for $\alpha > 0$, with probability one,
$\theta_{k,t} \rightarrow 0$ as $t \to \infty$ for each $k \in [K]$.
Moreover, for each $k \in [K]$,
\begin{equation}
\label{eq:thetaCLT}
(1 +\alpha)^{\frac{\alpha}{1+\alpha}} t^{\frac{\alpha}{1+\alpha}} \theta_{k,t} \stackrel{d}{\longrightarrow} \Beta\left(n_{k,0}, \,N - n_{k,0} \right) \quad \mbox{as } t \to \infty.
\end{equation}
\end{itemize}
\end{theorem}

\begin{proof}
(i) By \eqref{eq:sharet} and \eqref{eq:TP2}, it is easily seen that for each $k \in [K]$ and $t \ge 0$,
\begin{equation}
\label{pi}
\ex (\pi_{k,t+1} \,|\, \mathcal{F}_t) 
= \frac{n_{k,t}}{N_t} \left(1- \frac{1}{N_t^\al}\right)
+ \frac{n_{k,t}}{N_t+1} \, \frac{N_t- n_{k,t}}{ N_t^{1+\al}} 
+ \frac{n_{k,t}+1}{N_t+1} \,  \frac{n_{k,t}}{ N_t^{1+\al}}.
\end{equation}
Recognizing the first term on the right side of \eqref{pi}, $\frac{n_{k,t}}{N_t}=\pi_{k,t}$, while all other terms sum up to zero, 
we conclude that $(\pi_{k,t}, \, t \ge 0)$ is an $\mathcal{F}_t$-martingale. 
The convergence in \eqref{eq:mglecv} follows from the martingale convergence theorem (see e.g. \cite[Section 4.2]{Durrett}).
By Proposition \ref{prop:embed}, $(\pmb{n}_t, \, t \ge 0)$ is a time-changed P\'olya urn.
So the limiting shares $(\pi_{1,\infty}, \ldots, \pi_{K,\infty})$ have the same distribution as that of the P\'olya urn, which is $\Dir(n_{1,0}, \ldots, n_{K,0})$.

Let $(\pmb{n}^\dagger_t, \, t \ge 0)$ be the P\'olya urn with $n^\dagger_{k, 0} = n_{k,0}$, 
and $(\pi^\dagger_{k,t}, \ldots, \pi^\dagger_{K,t})$ be the corresponding shares. 
Set $Z \stackrel{d}{=} \Beta\left(n_{k,0}, N - n_{k,0} \right)$.
By \cite{GR13}, we have for each $k \in [K]$,
\begin{equation}
\label{eq:Steinrate}
d_W \left( \pi^\dagger_{k,t}, \, Z \right) = \mathcal{O} (t^{-1}) \quad \mbox{as } t \to \infty.
\end{equation}
Taking $\lambda < \lambda_{-}(\alpha)$, we get
\begin{align*}
d_W(\pi_{k,t}, \, Z) & \le \pr(N_t < \lambda t^{\frac{1}{1+\alpha}}) + d_W\left(\pi_{k,t} 1_{N_t \ge \lambda t^{\frac{1}{1+\alpha}}}, \, Z \right) \\
& \le C \pr(N_t < \lambda t^{\frac{1}{1+\alpha}}) + \sum_{s \ge \lambda t^{\frac{1}{1+\alpha}}} d_W((\pi_{k,t} | N_t = s), Z) \pr(N_t = s) \\
& = C \pr(N_t < \lambda t^{\frac{1}{1+\alpha}}) + \sum_{s \ge \lambda t^{\frac{1}{1+\alpha}}} d_W(\pi^\dagger_{k,s-N}, Z) \pr(N_t = s) \\
& \le C \exp(-C't^{\frac{1}{1+\alpha}}) + C''t^{-\frac{1}{1+\alpha}},
\end{align*}
where the last inequality follows from Theorem \ref{prop:2} and \eqref{eq:Steinrate}.
This yields the bound \eqref{eq:Warate}.

(ii) Applying the same derivation as in (\ref{pi}) but to $\theta_{k,t}$ instead, we have
\begin{equation*}
\ex (\theta_{k,t+1} \,|\, \mathcal{F}_t) 
= \frac{\pi_{k,t}}{N_t^{\al}} \left(1- \frac{1}{N_t^\al}\right)
+ \frac{N_t\pi_{k,t}}{(N_t+1)^{1+\al}}\cdot \frac{1- \pi_{k,t}}{ N_t^{\al}} 
+ \frac{N_t\pi_{k,t}+1}{(N_t+1)^{1+\al}} \cdot \frac{\pi_{k,t}}{ N_t^{\al}}.
\end{equation*}
The last two terms add up to $\frac{\pi_{k,t}}{N_t^\al (N_t+1)^\al} =\frac{\theta_{k,t}}{(N_t+1)^\al}$. 
Thus, we have
\begin{equation*}
\ex (\theta_{k,t+1} \,|\, \mathcal{F}_t) 
=\theta_{k,t} \left( 1-\frac{1}{N_t^\al}+ \frac{1}{(N_t+1)^\al}\right) \le \theta_{k,t},
\end{equation*}
i.e.  $(\theta_{k,t}, \, t \ge 0)$ is an $\mathcal{F}_t$-super-martingale for each $k$. 
Recall that $N_t^ \alpha \theta_{k,t} = \pi_{k,t}$ 
so $\theta_{k,t} \le N_t^{-\alpha}$ which converges to $0$ with probability one.
By (i), $N_t^\alpha \theta_{k,t}$ converges almost surely, and hence in distribution to $Z$. 
By Proposition \ref{prop:Nt}, $N_t/((1 +\alpha) t)^{\frac{1}{1+\alpha}}$ converges in probability to $1$.
We then apply Slutsky's theorem to get the convergence in \eqref{eq:thetaCLT}.
\end{proof}

\quad Several remarks are in order. 
Part (i) of Theorem \ref{thm:1} shows that the bidder shares form a martingale, and converges to a Dirichlet distribution (independent of $\alpha$).
This should be expected from the fact that the underlying bidder stakes $(\pmb{n}_t, \, t \ge 0)$ is a time-changed P\'olya urn; 
refer to Proposition \ref{prop:embed}. 
What's more revealing is the Wasserstein bound in \eqref{eq:Warate} between a bidder's share and its limit.
In fact, a matching lower bound can also be established (which we leave to the interested reader). 
Thus, the convergence rate of the bidder shares is exactly of order $t^{-\frac{1}{1+\alpha}}$.
(Also refer to Proposition \ref{thm:2} below 
for further discussions on the stability of the bidder shares when the initial stakes $N:=N_0$ is large.)

\quad Part (ii) of the theorem implies that each bidder's voting power decays to zero at rate {$t^{-\frac{\alpha}{1+\alpha}}$}. 
Or, equivalently, the reward rate is slowed down: it takes a time of order $\Theta(t^{\frac{\alpha}{1+\alpha}})$ 
for any bidder to be rewarded a new (unit of) stake.  
This enhances {security}, so that no bidder can manipulate or control the bidding/voting process;
{while the level of decentralization remains unchanged.}
This also means the principle of security in (\ref{security}) becomes ``easier'' to hold at large time $t$,
since (due to the network delay) $v \propto N_t^{-\alpha} \downarrow 0$ as $t \to \infty$.
On the other hand, if the reward is associated with transaction validation (which does not need to be), 
then the time required to validate a new block becomes uncontrolled in the long run.
A possible remedy is to dynamically
tune the parameter $\alpha$ over time, as detailed in the Appendix \ref{sc42}.

\section{Other Results, with PoS-Crypto Applications}
\label{sc3}

\quad In this section, we present more results associated with the $\texttt{Poly}(\al)$ model 
that are largely motivated by the application of PoS in cryptocurrency. 
There are two subsections: 
In the first one, \S\ref{sc31}, we study the the evolution of bidder shares when $N:=N_0$, the volume of initial stakes, is large. 
In \S\ref{sc32}, 
we study the additional feature of allowing the bidders to trade stakes
among themselves, focusing on the issue of trading incentive (or the lack thereof). 
We remark that the results in both subsections exhibit some type of phase transitions,
and are independent of the parametric value of $\al$, and in this sense, {\em universal}.

\subsection{Evolution of bidder shares and phase transitions}
\label{sc31}

As explained in the introduction, one key feature of the $\texttt{Poly}(\al)$ model is that
the reward rate, or the voting power (if the reward goes with validation work) $\theta_{k,t}$ of any bidder $k$ 
is different from $k$'s share $\pi_{k,t}$ of the total volume of stakes at $t$. 
We have seen from Theorem \ref{thm:1} (iii) that the reward rate or voting power is decreasing over time, 
which facilitates security. 
On the other hand, the evolution of the share $\pi_{k,t}$ over time, from its initial value $\pi_{k,0}$,
in both absolute and relative terms,
is an important issue for any individual bidder $k$.

\quad In the classical P\'olya urn setting, it is shown in \cite{RS21} that for a large bidder
with initial stake $n_{k,0} = \Theta(N)$, 
there is the stability in bidder share, in the sense that
\begin{equation*}
\pr(|\pi_{k,\infty} - \pi_{k,0}| > \varepsilon) \to 0 \quad \mbox{as } N \to \infty.
\end{equation*}
Furthermore, similar, albeit qualitatively different, results are revealed in 
 \cite{Tang22}, for small bidders (following the definition in part (ii) 
 of the corollary below).
 Here we focus on the ratio $\pi_{k,t}/\pi_{k,0}$. 
Since $\pi_{k,\infty} \stackrel{d}{=} \Beta(n_{k,0}, N - n_{k,0})$, 
 the results in \cite{Tang22} hold. 
 The following proposition
  is a refined version of \cite[Theorem 2.1]{Tang22}.
  
\begin{proposition}[Phase transitions of $\pi_{k,t}$]
\label{thm:2}
Let $N_0 = N$ be the total number of initial stakes.
Under the $\texttt{Poly}(\al)$ voting rule, we have
\begin{itemize}
\item[(i)]
For $n_{k,0} = f(N)$ such that $f(N) \to \infty$ as $N \to \infty$ (i.e. $\pi_{k,0} \gg 1/N$),
and for each $\varepsilon > 0$  sufficiently small
and each $t \ge 1$ or $t = \infty$,
\begin{equation}
\label{eq:coninter}
\pr \left(\left|\frac{\pi_{k, t}}{\pi_{k,0}} - 1\right| > \varepsilon \right) \le \frac{1}{\varepsilon^2 f(N)},
\end{equation}
which converges to $0$ as $N \to \infty$.
\item[(ii)]
For $n_{k,0} = \Theta(1)$ (i.e. $\pi_{k,0} = \Theta(1/N)$), 
there is the convergence in distribution 
\begin{equation}
\pi_{k,\infty}/\pi_{k,0} \stackrel{d}{\longrightarrow} \frac{1}{n_{k,0}} \gamma(n_{k,0}) \quad \mbox{as } N \to \infty,
\end{equation}
where $\gamma(n_{k,0})$ is a Gamma random variable with density $x^{n_{k,0}-1} e^{-x} 1_{x > 0}/\Gamma(n_{k,0})$.
Moreover, there is $C > 0$ (independent of $t$ and $N$) such that
\begin{equation}
\label{eq:Waspikt}
d_W\left(\frac{\pi_{k,t}}{\pi_{k,0}}, \, \frac{1}{n_{k,0}} \gamma(n_{k,0})\right) \le C\left( N^3 t^{-\frac{1}{1+\alpha}} +  \frac{1}{\sqrt{N}}\right).
\end{equation}
\end{itemize}
\end{proposition}

\begin{proof}
(i) Conditioning on $N_t$ and using the law of total variance, we get
\begin{equation*}
\var(\pi_{k,t}) = \frac{1 - N \ex(N_t^{-1})}{N+1} \pi_{k,0} (1 - \pi_{k,0}).
\end{equation*}
It suffices to apply Chebyshev's inequality to get the bound \eqref{eq:coninter}.

(ii) Note that
\begin{align*}
d_W\left(\frac{\pi_{k,t}}{\pi_{k,0}}, \, \frac{1}{n_{k,0}} \gamma(n_{k,0})\right) & \le  d_W\left(\frac{\pi_{k,t}}{\pi_{k,0}}, \, \frac{1}{\pi_{k,0}}\Beta(n_{k,0},  N - n_{k,0})\right) \\
& \qquad + d_W\left(\frac{1}{\pi_{k,0}}\Beta(n_{k,0},  N - n_{k,0}), \, \frac{1}{n_{k,0}} \gamma(n_{k,0}) \right)
\end{align*} 
A careful application of \cite{GR13} yields a refinement of \eqref{eq:Steinrate}: there is $C > 0$ such that
$d_W(\pi^\dagger_{k,t}, \, Z) \le \frac{C N^3}{t}$.
Adapting the argument in Theorem \ref{thm:1} yields
\begin{equation}
\label{est1h}
d_W\left(\frac{\pi_{k,t}}{\pi_{k,0}}, \, \frac{1}{\pi_{k,0}}\Beta(n_{k,0},  N - n_{k,0})\right) \le C'N^3 t^{-\frac{1}{1+\alpha}} \quad 
\mbox{for some } C' > 0.
\end{equation}
Next we claim that
\begin{equation}
\label{est2h}
d_W\left(\frac{1}{\pi_{k,0}}\Beta(n_{k,0},  N - n_{k,0}), \, \frac{1}{n_{k,0}} \gamma(n_{k,0}) \right) \le \frac{C''}{\sqrt{N}} \quad \mbox{for some } C'' > 0,
\end{equation}
which can be proved by elementary calculus.
Here we provide a sketch of proof. 
Set $n_{k,0} = 1$ for simplicity. 
Let $X \sim \gamma(1)$, 
and let $X'$ be the sum of $N-1$ independent $\gamma(1)$ random variables, independent of $X$.
By beta-gamma algebra, 
$\frac{X}{X+X'}$ has the same distribution as $\Beta(1, N-1)$.
Thus,
\begin{equation}
\label{eq:coupling1}
d_W(N\, \Beta(1, N-1), \gamma(1)) \le \mathbb{E}\left|\frac{NX}{X + X'}  - X\right|.
\end{equation}
By normal approximation, 
we have $\frac{X+X'}{N} = 1 + \frac{1}{\sqrt{N}} \mathcal{N}(0,1) + o(N^{-\frac{1}{2}})$ 
where $\mathcal{N}(0,1)$ is standard normal (see \cite{Rio09}).
Injecting into \eqref{eq:coupling1} yields the desired bound.
Finally, combining the estimates \eqref{est1h} and \eqref{est2h} gives the bound \eqref{eq:Waspikt}.
\end{proof}

\quad The proposition reveals a phase transition in the stability of shares, and identifies large and small bidders, in terms of the size of their stakes, according to the categories in the two parts.
A large bidder $k$ is guaranteed to have stability, 
in the precise sense characterized in (\ref{eq:coninter}),
that the share ratio $\pi_{k, t}/\pi_{k,0}$ concentrates at $1$, and converges to $1$ 
in probability when $N\to \infty$, 
for any $t \ge 1$ (including $t = \infty$). 
For small bidders, this is reversed:
the concentration inequality in \eqref{eq:coninter}
becomes the anti-concentration inequality:
\begin{equation}
\pr \left( \left| \frac{\pi_{k, \infty}}{\pi_{k,0}} - 1\right| > \varepsilon \right) > c \quad \mbox{for } c > 0 \mbox{ independent of $\varepsilon$},
\end{equation}
implying volatility.
The Wasserstein bound \eqref{eq:Waspikt} is new, 
and it indicates that the ratio $\pi_{k,t}/\pi_{k,0}$ approaches the limiting Gamma distribution with an $N^{-\frac{1}{2}}$ error
for $t \ge N^{\frac{7}{2}(1+\alpha)}$. 
However, we do not know whether the ``$N^3$ dependence'' in \eqref{eq:Waspikt} is tight
so the ratio $\pi_{k,t}/\pi_{k,0}$ may mix at a faster rate.

\subsection{Participation and trading}
\label{sc32}

So far, we have not considered the possibility of allowing the bidders to 
trade stakes (among themselves). 
In the classical P\'olya urn model ($\al=0$), it is shown in \cite{RS21} that 
under certain conditions (which enforce some notion of ``risk neutrality''), there
will be no incentive for any bidder to trade. 
Here, we extend that to the $\texttt{Poly}(\al)$ model, allowing $\al$ to take
any non-negative values. Furthermore, we allow a bidder-dependent risk-sensitivity (or risk-averse)
parameter $\delta_k$, and study the issue of incentive as it relates to $\delta_k$.

\quad In the new setting of allowing trading, we need to modify the problem formulation presented
at the beginning of \S\ref{sc2}.
First, for each $k \in [K]$,
let $\nu_{k,t}$ be the number of stakes that bidder $k$ will trade at time $t$.
Then, instead of \eqref{eq:TP1}, the number of stakes $n_{k,t}$ evolves as
\begin{equation}
\label{eq:TP3}
n_{k, t} = \underbrace{n_{k, t-1} + 1_{S_{k,t}}}_{n'_{k,t}} + \nu_{k,t},
\end{equation}
i.e. $n'_{k,t}$ denotes the number of stakes bidder $k$ owns in between 
time $t-1$ and $t$, excluding those traded in period $t$.

\quad Note that $\nu_{k,t}$ will be up to bidder $k$ to decide, as opposed to the random
event $S_{k,t}$ which is exogenous; 
in particular, $\nu_{k, t}$ can be negative (as well as positive or zero). 
We will elaborate more on this below, but note that  $\nu_{k,t}$ will be constrained such
that after the updating in (\ref{eq:TP3}) $n_{k, t}$ will remain nonnegative.

\quad Let $\{P_t, \, t \ge 0\}$ be the price process of each (unit of) stake,
 which is a stochastic process assumed to be independent of the 
randomness induced by the $\texttt{Poly}(\al)$ voting rule (specifically, the process $\{S_{k,t}\}$). 
Hence, we augment the filtration $\{\mathcal{F}_t\}_{t \ge 0}$ 
with that of the exogenous price process $\{P_t, \, t \ge 0\}$ to a new filtration denoted
$\{\mathcal{G}_t\}_{t \ge 0}$.
Note that the price process $P_t$ is also assumed as exogenous in 
\cite{RS21}. 
This assumption need not be so far off, as the crypto's price tends to be affected by 
 market shocks 
(such as macroeconomics, geopolitics, breaking news, etc) much more than by trading activities.
So, here we isolate the price of each stake from any bidder's trading impact.

\quad Let $b_{k,t}$ denote (units of) the risk-free asset that bidder $k$ holds at time $t$, 
and $r_{\tiny \mbox{free}}> 0$ the risk-free (interest) rate.
(Here, the risk-free asset is naturally the one that underlies the above 
price process.) 
As we are mainly concerned with the effect of exchanging stakes to each individual, 
we allow bidders to trade stakes only internally among themselves, but not risk-free assets between them.
Hence, each bidder has to trade risk-free asset with a third-party instead of trading that with another bidder.

\quad The decision for each bidder $k$ at $t$ is hence a tuple, $(\nu_{k,t}, b_{k,t})$.
Moreover, there is a terminal time, denoted $T_k \in \mathbb{N}_{+}$ (i.e., $T_k\ge 1$ is integer valued), 
by which time 
bidder $k$ has to sell all assets, including both any risk-free asset and any stakes owned at that time, 
and leave the system.
$T_k$ can either be deterministic or random. 
In the latter case, assume it has a finite expectation,
and is either adapted to $\{\mathcal{G}_t\}_{t \ge 0}$,
or independent of all other randomness (in which case augment $\{\mathcal{G}_t\}$ accordingly).  
We also allow bidder $k$ to leave the system and liquidate prior to $T_k$ at a stopping time $\tau_k$ 
relative to $\{\mathcal{G}_t\}_{t \ge 0}$.
Thus, bidder $k$ will also decide at which time $\tau_k$ to stop and exit.
To simplify the notation, we abuse $\tau_k$ for $\tau_k \wedge T_k$, the minimum of $\tau_k$ and $T_k$.

 
\quad Let $c_{k,t}$ denote the (free) cash flow  (or, ``consumption'') of bidder $k$ at time $t$, i.e.,
\begin{equation*}
 c_{k,t} = (1+r_{\tiny \mbox{free}}) b_{k, t-1} - b_{k, t}  - \nu_{k,t}P_t, 
 \quad \forall 1\le t< \tau_k; \tag{C1}
\end{equation*}
with 
\begin{equation*}
b_{k,0} = 0, \; b_{k,t} \ge 0,\quad
 0 \le n_{k,t} = n'_{k,t} + \nu_{k,t} \le N_t, \quad \forall 1\le t< \tau_k; \tag{C2}
\end{equation*} 
and
\begin{equation*}
c_{k, \tau_k} = (1+ r_{\tiny \mbox{free}}) b_{k, \tau_k- 1} + n'_{k, \tau_k} P_{\tau_k}, 
 \quad \mbox{ and } \nu_{k,\tau_k}=b_{k, \tau_k} = 0.  \tag{C3} 
\end{equation*}

\quad Observe that the equation in (C1) simply defines what's available for ``consumption'' 
in period $t$.
It is simply an accounting or budget constraint on the cash flow.
 The requirements in (C2) are all in the spirit of disallowing shorting, on both components
 of the decision, the free asset $b_{k,t}$ and 
the traded stakes $\nu_{k,t}$. In particular the latter is constrained such that 
$\nu_{k,t} \ge -n'_{k,t}$ (following $n_{k,t} \ge 0$) i.e., bidder $k$ cannot sell more
than what's in possession at $t$; it also ensures that 
no bidder can own a number of stakes beyond the current total volume ($n_{k,t} \le N_t$).
(C3) specifies how the assets are liquidated at the exit time $\tau_k$: 
both $\nu_{k,\tau_k}$ and $b_{k, \tau_k}$ will be set at zero, 
and all remaining stakes $n'_{k, \tau_k}$ liquidated (cashed out at $P_{\tau_k}$ per unit).

\quad Denote bidder $k$'s decision (process) or ``strategy''  as $\tau_k$ and $(\nu,b):=\{(\nu_{k,t}, b_{k,t}), \, 1\le t\le \tau_k\}$. 
The objective of bidder $k$ is 
\begin{equation}
\label{eq:OPT}
U^*_k := \max_{\tau_k, (\nu ,b)}  U_k :=
\max_{\tau_k, (\nu ,b)}  \ex \left(\sum_{t=1}^{\tau_k} \delta_k^{t} c_{k,t} \right), 
\quad \text{subject to (C1), (C2), (C3)} ;
\end{equation}
where $\delta_k \in (0,1]$ is a discount factor, a given parameter measuring the risk sensitivity of bidder $k$.
Clearly, bidder $k$'s objective is to maximize a utility that is just 
the present value of $k$'s total cash flow cumulated up to $T_k$.

\quad We need to introduce two more processes 
that are related and central to understanding the dynamics of the system in the presence of trading.  
The first one is $\{M_{t}, \, t\ge 1\}$, where
$M_t := N_t P_t $ 
denotes the market value of the volume of stakes at time $t$.
The second one is  $\{\Pi_{k,t}, \, t\ge 0\}$, for each bidder $k$, defined as follows:
\begin{equation}
\label{eq:Pi}
\Pi_{k,0} := n_{k,0} P_0, \quad \mbox{and} \quad 
\Pi_{k,t} := \delta^t_k n'_{k,t} P_t - \sum_{j=1}^{t-1} \delta_k^j \nu_{k, j} P_j , \quad t \ge 1;
\end{equation}
where $n'_{k,t+1}$ follows (\ref{eq:TP3}).
Note the two terms that define $ \Pi_{k,t}$ are the discounted present values, respectively, 
of $k$'s pre-trading stakes ($n'_{k,t}$) and of the return from $k$'s trading (cumulated up to $t-1$).

\quad The connection between $\{M_t\}$ and $\{\Pi_{k,t}\}$ is presented in the following lemma,
which reveals that their incremental gains (per time period) are proportional:
each increment of $\Pi_{k,t}$ is a $\pi_{k,t}$ 
fraction of  the corresponding increment of $M_t$. 
In other words, $\pi_{k,t}$ not only represents bidder $k$'s share of the total volume 
of stakes, it also represents $k$'s share of the system's market value, with or without trading.

\begin{lemma}
\label{lem:utility}
Under the $\texttt{Poly}(\al)$ voting rule, along with the trading specified above, we have
\begin{equation}
\label{eq:ut}
\ex(\Pi_{k,t+1} \,|\, \mathcal{G}_t) - \Pi_{k,t} 
= \delta_k^{t+1}\pi_{k,t} \ex(M_{t+1} \,|\, \mathcal{G}_t) - \delta_k^t \pi_{k,t} M_t .
\end{equation}
\end{lemma}

\begin{proof}
First, by \eqref{eq:TP3} and \eqref{eq:TP2}, along with $\pi_{k,t}= n_{k,t}/N_t$, we have
\begin{equation}
\label{eq:mag}
\ex(n'_{k,t+1} \,|\, \mathcal{F}_t) = n_{k,t} (1 + N_t^{-(1+\al)}) 
= \frac{n_{k,t}}{N_t} (N_t + N_t^{-\al}) = \pi_{k,t} \ex(N_{t+1} \,|\, \mathcal{F}_t).
\end{equation}
Next, from \eqref{eq:Pi}, we have
\begin{equation}
\label{eq:Pi1}
\Pi_{k,t+1}- \Pi_{k,t} = \delta^{t+1}_k n'_{k,t+1} P_{t+1} - \delta^t_k n'_{k,t} P_t 
- \delta_k^t \nu_{k, t} P_t,  \quad t \ge 1.
\end{equation}
Furthermore, as the price process $(P_t, \, t \ge 0)$ is independent of $\mathcal{F}_t$, 
we have
\begin{align*}
\ex(n'_{k,t+1} P_{t+1} \,|\, \mathcal{G}_t) &= \ex(\ex(n'_{k, t+1} \,|\, \mathcal{F}_t) P_{t+1} \,|\, \mathcal{G}_t) \\
& \stackrel{\eqref{eq:mag}}{=} \pi_{k,t} \ex(N_{t+1}P_{t+1} \,|\, \mathcal{G}_t) = \pi_{k,t} \ex(M_{t+1} \,|\, \mathcal{G}_t).
\end{align*}
This, along with \eqref{eq:Pi1} yields the desired expression in \eqref{eq:ut}, along with 
$n_{k,t}= n'_{k,t}+\nu_{k,t}$, $n_{k,t} = \pi_{k,t} N_t$, and $M_t=N_tP_t$. 
\end{proof}

\quad The process $\{\Pi_{k,t}\}$ also connects to the utility $U_k$ in \eqref{eq:OPT}.
To see this, summing up both sides of (C1) and (C3) over $t$ (along with $b_{k,0} = 0$ in (C2)), we have
\begin{equation}
\label{eq:ut1}
\sum_{t \le \tau_k} \delta^{t}_k c_{k,t} = \sum_{t \le \tau_k} \delta^{t}_k c_{k,t} = \delta_k^{\tau_k} n'_{\tau_k} P_{\tau_k}
-\sum_{t =1}^{\tau_k-1} \delta^{t}_k \nu_{k,t} P_t
+\sum_{t =1}^{\tau_k-1}\delta_k^t \left[(1+r_{\tiny \mbox{free}}) \delta_k - 1\right] b_{k,t}. 
\end{equation}
Observe that the first two terms on the RHS are equal to $\Pi_{k,\tau_k}$,  so we can rewrite the above as
follows (after taking expectations on both sides), emphasizing the exit time $\tau_k$ and the strategy $(\nu,b)$, 
\begin{equation}
\label{eq:uk}
U_k (\tau_k, \nu,b) 
=\ex \left[ \Pi_{k,\tau_k} (\nu)\right] + \ex\left( \sum_{t =1}^{\tau_k -1}\delta_k^t \left[(1+r_{\tiny \mbox{free}}) \delta_k - 1\right] b_{k,t}\right);  
\end{equation}
hence, the RHS above is {\it separable}: 
the first term 
depends on $(\nu)$ only while the second term, the summation,  on $(b)$ only.
Furthermore,  the second term is $\le 0$ provided
$(1+r_{\tiny \mbox{free}}) \delta_k \le 1$ (which is the condition (a) assumed in Theorem \ref{thm:3} below),  
along with $b$ being non-negative, part of the feasibility in (C2). 
In this case, we will have $U_k\le\ex ( \Pi_{k,\tau_k}(\nu))$, 
which implies $U^*_k \le\max_{\tau_k, \nu} \ex( \Pi_{k,\tau_k}(\nu))$,
with equality holding when $b_{k,t}=0$ for all $t=1, \dots, \tau_k$.


\quad We are now ready to present the main result 
regarding the utility maximization problem in  \eqref{eq:OPT}. 
A quick word on the parameter $r_{\tiny \mbox{cryp}}$ that will appear prominently
in Theorem \ref{thm:3} below. 
Simply put, it is the rate (expected rate of return) associated with 
each stake (e.g., a unit of some cryptocurrency), i.e., it is
the counterpart of  $r_{\tiny \mbox{free}}$, the rate for the risk-free asset.
We will elaborate more on the two rates after proving the theorem.
In the theorem, two strategies are singled out:
the ``buy-out'' strategy, in which bidder $k$ buys up all stakes available at time $1$, and then 
participates in the bidding process until the end;
and the ``non-participation'' strategy, in which bidder $k$ turns all $n_{k,0}$ stakes into cash,
and then never participates in either bidding or trading for all $t\ge 1$.
Note that the non-participation strategy is executed at $\tau_k=0$; as such, it complements the feasible class, which is  
for $ \tau_k \ge 1$ and presumes participation.
The buy-out strategy clearly belongs to the feasible class.

\begin{theorem}[Buy-out strategy versus non-participation]
\label{thm:3}
Assume the following two conditions: 
\begin{equation}
\label{condM}
{\rm (a)}\;\; \delta_k(1+r_{\tiny \mbox{free}}) \le 1 \quad{\rm and}\quad
{\rm (b)} \;\; \ex(M_{t+1} \,|\, \mathcal{G}_t) = (1+r_{\tiny \mbox{cryp}}) M_t.
\end{equation}
Then, under the $\texttt{Poly}(\al)$ voting rule, the following results will hold.

\quad First, with condition (a), the maximal utility
$U^*_k$ is achieved by setting $b_{k,t}=0$ for all $t=1, \dots, T_k$; i.e.,
$U^*_k =\max_\nu \ex( \Pi_{k,T_k})$.

\quad In addition, all three parts of the following will hold. 


\begin{enumerate}

\item[(i)] If
 $\delta_k  (1+r_{\tiny \mbox{cryp}}) \le 1$,
then 
any feasible strategy 
will provide no greater utility for bidder $k$ than the non-participation strategy, 
 i.e., $U^*_k \le n_{k,0}P_0$.

\item[(ii)] If
$ \delta_k  (1+r_{\tiny \mbox{cryp}}) \ge 1$,
then any feasible strategy will provide no greater utility for bidder $k$ than the buy-out strategy.
In this case, bidder $k$ will buy all available stakes at time $1$, and participate in the bidding process until the terminal time $T_k$.

\item[(iii)] If
 $\delta_k (1+r_{\tiny \mbox{cryp}})=1$, 
then, 
bidder $k$ is indifferent between
the non-participation and the buy-out strategy with any exit time, both of which will
provide no less utility than any feasible strategy.
In other words,  
all strategies achieve the same utility (which is $\Pi_{k,0}=n_0P_{k,0}$).

\end{enumerate}

\quad Moreover, when
$\delta_k =\delta:= (1+r_{\tiny \mbox{cryp}})^{-1}$ for all $k$, 
then no bidder will have any incentive to trade. Consequently,
the long-term behaviors (of $N_t$, $\pi_{k,t}$ and $\theta_{k,t}$) characterized  
in Proposition \ref{prop:Nt}, Theorem \ref{thm:1} and Proposition \ref{thm:2} will hold.
\end{theorem}




\begin{proof}
That 
$U^*_k = \max_{\tau_k, \nu} \ex( \Pi_{k,\tau_k})$ (with $b_{k,t}$ being set to $0$ for all $t$), 
under condition (a) in \eqref{condM}, has already been established in the discussions following \eqref{eq:uk}. 
So, it suffices to prove the three parts (i)-(iii).  

\smallskip

(i) Applying the given
condition (b) in \eqref{condM}, along with the assumed inequality  $\delta_k(1+r_{\tiny \mbox{cryp}}) \le 1$,
to the RHS of the equation in \eqref{eq:ut} will make it $\le 0$; i.e.,, 
$\{\Pi_{k,t}\}$ is a $\mathcal{G}_t$-super-martingale, implying $\ex(\Pi_{k, \tau_k}) \le \Pi_{k,0}$. 
Since $\Pi_{k,0}$ is independent of $\nu$, we have 
\begin{equation}
\label{eq:ut2}
U^*_k = \max_{\tau_k, \nu} \ex(\Pi_{k, \tau_k}) \le \Pi_{k,0} = n_{k,0}P_0,
 \end{equation}

(ii) 
With the assumed inequality $\delta_k(1+r_{\tiny \mbox{cryp}}) \ge 1$,
$\{\Pi_{k,t}\}$ now becomes a $\mathcal{G}_t$-sub-martingale; and
hence, the inequality below, 
\begin{equation}
\label{eq:ut200}
\ex(\Pi_{k, T_k}) \ge \ex(\Pi_{k,\tau_k}) \ge \Pi_{k,0} = n_{k,0} P_0.
\end{equation}

To identify the optimal trading strategy $\{\nu^*_{k, j}\}_{j \le T_k - 1}$, we use 
backward induction (dynamic programming). 
To optimize $\nu_{k, T_k - 1}$, observe
\begin{align*}
& \quad \ex( \delta_k^{T_k} n'_{k, T_k} P_{T_k} - \delta_k^{T_k - 1} \nu_{k, T_k - 1} P_{T_k - 1}  \,|\, \mathcal{G}_{T_k - 1} ) \\
& = \delta_k^{T_k} (n'_{k, T_k - 1} + \nu_{k, T_{k}-1})  \left(1 + N_{T_k - 1}^{-\alpha - 1} \right) \ex(P_{T_k} | \mathcal{G}_{T_k - 1}) - \delta_k^{T_k - 1} \nu_{k, T_k - 1} P_{T_k - 1} \\
& =  \delta_k^{T_k} n'_{k, T_k - 1} N_{T_k - 1}^{-1} \ex(N_{T_k} P_{T_k} | \mathcal{G}_{T_k - 1}) 
+ \delta_k^{T_k - 1} \left(\delta_k N_{T_k - 1}^{-1} \ex( N_{T_k}P_{T_k} | \mathcal{G}_{T_k - 1}) - P_{T_k - 1} \right) \nu_{k, T_k - 1}\\
& = \delta_k^{T_k} n'_{k, T_k - 1} N_{T_k - 1}^{-1} \ex(M_{T_k}  | \mathcal{G}_{T_k - 1}) 
+ \delta_k^{T_k - 1} \left(\delta_k N_{T_k - 1}^{-1} \ex( M_{T_k} | \mathcal{G}_{T_k - 1}) - P_{T_k - 1} \right) \nu_{k, T_k - 1}
\end{align*}
which is linear in $\nu_{k, T_{k-1}}$.
By assumed condition (b) in \eqref{condM}, we have
\begin{equation*}
\delta_k N_{T_k - 1}^{-1} \ex( M_{T_k} | \mathcal{G}_{T_k - 1}) - P_{T_k - 1}
\ge \left( \delta_k (1 + r_{\tiny \mbox{cryp}}) - 1 \right) P_{T_k - 1} \ge 0.
\end{equation*}
Thus, $(\nu^*_{k, T_k -1} \,|\, \mathcal{G}_{T_k-1})  =  N_{T_k- 1} - n'_{k, T_k - 1}$, following the (binding) constraint in (C2). 
That is, bidder $k$'s optimal strategy at the penultimate time $T_k - 1$ is to buy all available stakes at that time. 
Going backward, we have $(\nu^*_{k,j} \,|\, \mathcal{G}_{j}) = N_{k,j} - n'_{k,j}$ for $j \ge 1$.
Thus, the optimal trading strategy is $\nu^*_{k,1} = N_1 - n'_{k,1}$, $\nu^*_{k,2} = \cdots = \nu^*_{k, T_k - 1} = 0$.

(iii) Under the assumed equality $\delta_k(1+r_{\tiny \mbox{cryp}}) = 1$, 
$\{\Pi_{k,t}\}$ is a $\mathcal{G}_t$-martingale;
hence, the inequality in \eqref{eq:ut2} now holds as equality, i.e.,
$U^*_k = \Pi_{k,0} = n_{k,0}P_0$.
 Thus. all strategies lead to the optimal utility, including  any feasible strategy (in particular, the no-trading strategy)
  and the non-participation strategy. 

The ``moreover'' part of the theorem is immediate.
\end{proof}

\quad In what remains of this section, we make a few remarks on Theorem \ref{thm:3}, 
in particular, to motivate and explain its required conditions. 
First, the rate $r_{\tiny \mbox{cryp}}$ is determined by condition (b), the second equation in  \eqref{condM}.
As such, it should be distinct from $r_{\tiny \mbox{free}}$, the latter being associated with a risk-free asset.
For all practical purpose, we can assume $r_{\tiny \mbox{crpt}}\ge r_{\tiny \mbox{free}}$, even though
this is not assumed in the theorem. When this relation does hold, then condition (a) will become
superfluous in cases (i) and (iii).

\quad Second, the discount factor $\delta_k$ in the utility objective in \eqref{eq:OPT}, 
a parameter that measures bidder $k$'s sensitivity towards risk, 
plays a key role in characterizing phase transitions in terms of $\delta_k(1+r_{\tiny{\mbox{cryp}}})$. 
In case (i), the inequality $\delta_k\le 1/(1+r_{\tiny \mbox{cryp}})$ implies bidder $k$ is seriously risk-averse; 
and this is reflected in $k$'s non-participation strategy.
In case (ii), the inequality holds in the opposite direction, implying bidder $k$ is lightly risk-averse or even a risk taker.
Accordingly, $k$'s strategy is to aggressively sweep up all the available stakes to reach monopoly, 
and participate (but not trade) until the terminal time.
Also note in this case, the non-participation strategy will provide less (no greater) utility for bidder $k$ than
the ``no-trading'' strategy, and certainly no greater utility than
the buy-out strategy.
In case (iii), the inequality becomes an equality $\delta_k = 1/(1+r_{\tiny \mbox{cryp}})$, and $\{\Pi_{k,t}\}$ becomes a martingale.
Consequently, bidder $k$ is indifferent between non-participation and participation, 
and in the latter case, indifferent to all (feasible) strategies, including 
the buy-out (and the no-trading) strategy. Indeed, the  equality $\delta_k = 1/(1+r_{\tiny \mbox{cryp}})$
is both necessary and sufficient for the no-trading strategy.
This equality also has the effect to force all participating bidders to have the same risk sensitivity.

\quad In contrast, in \cite{RS21}, there is a single rate $r_{\tiny \mbox{free}}$, or equivalently.
$r_{\tiny \mbox{cryp}}=r_{\tiny \mbox{free}}$ is assumed, which seems difficult to justify,
since in most applications (cryptocurrency in particular) $r_{\tiny \mbox{cryp}}$ will be significantly larger than 
$r_{\tiny \mbox{free}}$.
Moreover,  there is also a single risk sensitivity for all bidders, which is set at 
 $\delta=1/(1+r_{\tiny \mbox{free}})$. Thus,
\cite{RS21} is limited to the martingale case only, 
reaching the same conclusion as our case (iii), 
that all feasible strategies, buy-out included, yields the same (expected) utility. 
As there is no stopping decision and super- or sub-martingale cases
in \cite{RS21}, non-participation does not come up at all, 
neither do notions like risk-averse or risk-seeking.   



\quad The last point we want to emphasize is 
 that the two conditions in \eqref{condM} play very different roles.
  As evident from the proof of Theorem \ref{thm:3},  condition (b)
  makes $\{\Pi_{k,t}\}$ a super- or sub-martingale or a martingale,
  according to bidder $k$'s risk sensitivity as specified by the inequalities and equality applied to $\delta_k$
  (along with $r_{\tiny \mbox{free}}$) in the three cases.
  Yet, to solve the maximization problem in \eqref{eq:OPT}, 
  $\{\Pi_{k,t}\}$ needs to be connected to the utility; and this is the role played by
  condition (a), under which it is necessary (for optimality) to set $b_{k,t}=0$ for all $t\ge 1$, 
and applicable to all three cases in Theorem \ref{thm:3}.
 In this sense, condition (a) alone solves half of the maximization problem, the $b_{k,t}$ half of the strategy.  
  In fact, it's more than half, as the optimal $\nu$ strategy is only needed in the sub-martingale case; and even 
  there, condition (a) pins down the fact that to participate (even without trading) is better than non-participation. 

\quad Note that Theorem \ref{thm:3} can be readily extended. 
For instance, the rates $r_{\tiny \mbox{cryp}}(t)$ and $r_{\tiny \mbox{free}}(t)$ can vary over the time. 
In this case, it suffices to modify the conditions in case (i) to 
\begin{equation*}
 \left(1+ \sup_{t < T_k} r_{\tiny \mbox{cryp}}(t)\right) \delta_k \le 1 \quad \mbox{and} \quad \left(1+\sup_{t < T_k} r_{\tiny \mbox{free}}(t)\right) \delta_k \le 1,
\end{equation*}
the conditions in case (ii) to
 \begin{equation*}
\left(1+ \inf_{t < T_k} r_{\tiny \mbox{cryp}}(t)\right) \delta_k \ge 1\quad \mbox{and} \quad \left(1+\sup_{t < T_k} r_{\tiny \mbox{free}}(t)\right) \delta_k \le 1,
\end{equation*}
and the conditions in case (iii) to 
\begin{equation*}
\delta_k = (1 + r_{\tiny \mbox{cryp}})^{-1} \quad \mbox{and} \quad \sup_{t < T_k} r_{\tiny \mbox{free}}(t) \le r_{\tiny \mbox{cryp}}, \quad
\mbox{with } r_{\tiny \mbox{cryp}}  \mbox{ being constant}.
\end{equation*}
Then, Theorem \ref{thm:3} will continue to hold. 
We can also include a processing cost $\kappa > 0$ that any bidder selected by the $\texttt{Poly}(\al)$ mechanism will pay to receive the reward. (This corresponds to the ``mining'' cost to validate the block.) 
In this case, the budget constraint (C1) is modified by 
adding a term $- \kappa 1_{S_{k,t}}$ to the right side of the equation, 
and the same applies to the liquidation constraint (with $t$ replaced by $T_k$).
Condition (b) in \eqref{condM} is modified to 
$\ex(M_{t+1} \,|\, \mathcal{G}_t) = (1+r_{\tiny \mbox{cryp}}) M_t + \kappa$.

\section{Conclusions}
\label{sc5}

\quad We have proposed in this study a new $\pol$ voting rule that is more 
general than the traditional voting rule (which is linear, corresponding to $\al=0$).
More importantly, the $\pol$ voting rule distinguishes voting power from voter share, and hence
decouples the two. 
 
\quad Applying the $\pol$ voting rule to the PoS protocol, where the voters are the bidders
(competing for ``rewards,'' or validation of new blocks), 
we show this decoupling will enhance {security}, a key objective of the PoS protocol.
Specifically, we prove that while bidder shares form a martingale process that will converge to a Dirichlet distribution, 
each bidder's voting power is a super-martingale 
that decreases to zero over time. 
For both limiting results, we explicitly characterize their rate of convergence as well.
Furthermore, we show a phase transition in the stability of bidder shares 
in terms of each bidder's initial share relative to the total in the system.
We also study the issue of a bidder's risk sensitivity when trading is allowed,
and provide conditions under which a bidder will have no incentive to participate in the bidding process, 
or, if participate, will forgo trading.

\quad In the Introduction, we mentioned two general approaches to enhance security in the PoS protocol:
adjust the amount of reward over time and slow down the voting process; and the current study focuses on the latter while
keeping the reward constant.
It is possible to pursue a combination of both approaches, i.e., adjusting the size of reward dynamically over time
in the same manner as adjusting $\alpha$ (for the latter, refer to \S\ref{sc42} below). 
In another direction, it is also possible to study the trading problem in \S\ref{sc32} under 
a suitable market impact model, where the price process $P_t$ will be impacted by trading activities;
for instance, 
 a mean-field PoS model with linear impact (and transaction costs). 
 
\bigskip

\textbf{Acknowledgments}:
We thank anonymous referees for helpful suggestions which improve the presentation of the paper. 
W.\ Tang gratefully acknowledges financial support through NSF grants DMS-2113779 and DMS-2206038,
and through a start-up grant at Columbia University.
David Yao's work is part of a Columbia-CityU/HK collaborative project that is supported by 
InnoHK Initiative, The Government of the HKSAR and the AIFT Lab.

\section{Appendix}

\subsection{Improvement on $\lambda_{\pm}(\al)$}
\label{sc41}

Theorem \ref{prop:2} proves large-deviation bounds on $N_t$.
However, it does not cover the whole range. 
It remains open to prove such bounds
in the range $(\lambda_-(\alpha), \lambda_+(\alpha))$;
and once proved, the result will also imply the almost sure convergence of $N_t/t^{\frac{1}{1 + \alpha}}$ as $t \to \infty$.

\quad Here we provide a way to (slightly) improve the values of $\lambda_{\pm}(\al)$ in Theorem \ref{prop:2}.
To simplify the presentation, we consider $\al = 1$ (quadratic voting rule) with $\lambda_{-}(1) \approx 0.56$ and 
$\lambda_{+}(1) \approx 2.51$.
The idea relies on a multi-scale analysis by splitting the interval $[0,t]$ into $[0,t/2]$ and $[t/2, t]$,
and the goal is to upper bound $\mathbb{P}(N_t = \lambda \sqrt{t})$ for $\lambda > 0$.
In the sequel, we neglect the polynomial factors and only focus on the exponential terms.
Note that
\begin{equation*}
\pr(N_t = \lambda \sqrt{t}) = \sum_{k \le \lambda \sqrt{t}} \binom{t/2}{k} \binom{t/2}{\lambda \sqrt{t} - k}
\frac{1}{(\lambda \sqrt{t} )!} \underbrace{\left(1 - \frac{1}{k} \right)^{t/2 - k} \left( 1 - \frac{1}{\lambda \sqrt{t}}\right)^{t/2 + k - \lambda \sqrt{t}}}_{(a'')}.
\end{equation*}
Next we split the range of $k \le \lambda \sqrt{t}$ into $S_1:=\{k \le a \sqrt{t}\} \cup \{k \ge (\lambda -a) \sqrt{t}\}$,
and $S_2: = \{a \sqrt{t} < k < (\lambda -a) \sqrt{t}\}$ with $a < \frac{\lambda}{2}$.
For $k \in S_1$, we simply bound the term (a) by $\left( 1 - \frac{1}{\lambda \sqrt{t}}\right)^{t/2  - \lambda \sqrt{t}}$
while for $k \in S_2$ we bound the term ($a''$) by
$\left(1 - \frac{1}{(\lambda - a) \sqrt{t}} \right)^{t/2 - (\lambda - a) \sqrt{t}} \left( 1 - \frac{1}{\lambda \sqrt{t}}\right)^{t/2 - a\sqrt{t}}$.
Consequently,
\begin{multline*}
\pr(N_t = \lambda \sqrt{t}) \le 
\underbrace{\left(\sum_{k \in S_1} \binom{t/2}{k} \binom{t/2}{\lambda \sqrt{t} - k} \right) \frac{1}{(\lambda \sqrt{t} )!} \left( 1 - \frac{1}{\lambda \sqrt{t}}\right)^{t/2  - \lambda \sqrt{t}}}_{(b'')} \\
+ \underbrace{\left(\sum_{k \in S_2} \binom{t/2}{k} \binom{t/2}{\lambda \sqrt{t} - k} \right) \frac{1}{(\lambda \sqrt{t} )!} \left(1 - \frac{1}{(\lambda - a) \sqrt{t}} \right)^{t/2 - (\lambda - a) \sqrt{t}} \left( 1 - \frac{1}{\lambda \sqrt{t}}\right)^{t/2 - a\sqrt{t}}}_{(c'')}.
\end{multline*}
Using Stirling's formula, we get exponential bounds for the terms ($b''$) and ($c''$):
\begin{equation}
\label{eq:bceq}
\begin{aligned}
& (b'') \sim \exp\left( (-\lambda \log 2 + 2 \lambda - a \log a - (\lambda -a) \log(\lambda - a) - \lambda \log \lambda - \frac{1}{\lambda}) \sqrt{t} \right), \\
& (c'') \sim \exp\left( (2 \lambda - 2 \lambda \log \lambda - \frac{1}{2 \lambda} - \frac{1}{2(\lambda - a)}) \sqrt{t}\right).
\end{aligned}
\end{equation}
By equating the two coefficients before $\sqrt{t}$ in \eqref{eq:bceq}, we have
\begin{equation*}
-\lambda \log 2 + 2 \lambda - a \log a - (\lambda -a) \log(\lambda - a) - \lambda \log \lambda - \frac{1}{\lambda})
= 2 \lambda - 2 \lambda \log \lambda - \frac{1}{2 \lambda} - \frac{1}{2(\lambda - a)}.
\end{equation*}
By letting $a = \theta \lambda$ with $\theta < \frac{1}{2}$, the above equation yields
\begin{equation}
\label{thetaha}
\lambda = \sqrt{\frac{\theta}{2(1 - \theta)(\log 2 + \theta \log \theta + (1- \theta) \log(1 - \theta))}}.
\end{equation}
and the coefficient before $\sqrt{t}$ is
\begin{equation}
\label{flambda}
f(\lambda) = 2 \lambda \log \lambda - 2 \lambda + \frac{1}{2 \lambda} + \frac{1}{2(1 - \theta) \lambda},
\end{equation}
where $\theta$ is specified by \eqref{thetaha}.
By injecting the expression \eqref{thetaha} into \eqref{flambda},
$f$ is a function of $\theta$.
It is easy to see that $f(\theta)$ has only one root on $(0,1/2)$ which is approximately $0.1575$, and 
$\lambda_{-}(1)$ is improved numerically to from $0.56$ to $0.60$.
Similarly, the value of $\lambda_{+}(1)$ is improved numerically from $2.51$ to $2.44$.

\quad We can continue this procedure, for instance to split $[0,t]$ into $[0,t/3]$, $[t/3, 2t/3]$ and $[2t/3, t]$, and so on to get better and better numerical values of $\lambda_{-}(1)$ and $\lambda_{+}(1)$.
However, it is not clear whether this approach will eventually get all the way to the threshold $\sqrt{2} \approx 1.41$.
We conjecture that the exponential deviation holds right off the threshold $(1+\al)^{\frac{1}{1+\al}}$,
which is supported by the numerical experiments; refer to Figure \ref{fig:1}.

\begin{figure}[htb]
    \centering
\begin{subfigure}{0.4\textwidth}
  \includegraphics[width=\linewidth]{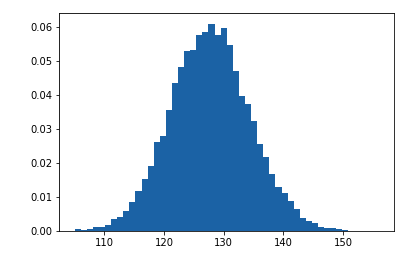}
  \caption{Histogram of $N_{8000}$ on MC simulation of $20000$ samples.}
\end{subfigure}\hfil
\begin{subfigure}{0.4\textwidth}
  \includegraphics[width=\linewidth]{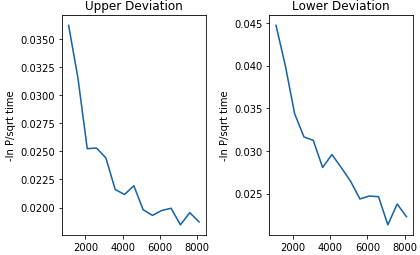}
  \caption{$x$-axis: $t \in \{1000, 1500, \ldots, 8000\}$;
$y$-axis: $-\ln \pr(N_t > \sqrt{2.2\, t})/\sqrt{t}$ (left) and $-\ln \pr(N_t < \sqrt{1.8 \, t})/\sqrt{t}$ (right)
on MC simulation of $20000$ samples.}
\end{subfigure}\hfil
\caption{Volume of stakes $N_t$ with $N_0 = 5$ and $\al = 1$ (quadratic voting).}
\label{fig:1}
\end{figure}

\subsection{Control of voting powers}
\label{sc42} 

As proved in Theorem \ref{thm:1}, the reward rate $\theta_{k,t}$ decays at rate $\Theta(t^{-\frac{\alpha}{1 + \alpha}})$.
If the reward is associated with the validation of a new block, then 
the duration between two consecutive validations (called ``block time'' below) 
will increase (and uncontrolled) over time. 
For instance, 
set $\alpha = 1$ (quadratic voting rule), and $T = 10^7$ seconds ($\approx 4$ months).
Then, the duration required to see the next block at time $T$ is approximately
\begin{equation*}
10 \mbox{ seconds } \times (10^7/10)^{\frac{1}{2}} = 10^4 \mbox{ seconds } \approx 3 \mbox{ hours},
\end{equation*}
which is even much longer than the $10$ minutes block time of Bitcoin.
(The block time is $10$ seconds in Ethereum, see e.g. \cite{BV14}.)

\quad One possible (and practical) solution 
 is to dynamically ``tune'' the parameter $\alpha$ over time. 
Specifically, let $\kappa$ denote a threshold for the expected number of rounds of 
bidding/voting between two validated blocks. 
Then,
\begin{itemize}[itemsep = 3 pt]
\item
set $\alpha = \alpha_0 > 0$, and apply the $\texttt{Poly}(\alpha_0)$ scheme up to round $\kappa^{1 + \alpha_0^{-1}}$;
\item
set $\alpha = \alpha_1 < \alpha_0$, 
and apply the $\texttt{Poly}(\alpha_1)$ scheme up to round $\kappa^{1 + \alpha_1^{-1}}\ldots$ and so on.
\end{itemize}
Here $\kappa, \alpha_0, \alpha_1, \ldots$ are user-defined hyper-parameters. 
To illustrate, by setting $\kappa = 50$ rounds ($\approx 10$ minutes in Ethereum) and $\alpha_k = (1+k)^{-1}$ for $k \ge 0$,
\begin{itemize}[itemsep = 3 pt]
\item[-]
Apply the $\texttt{Poly}(1)$ scheme up to round $50^2 \approx 7$ hours;
\item[-]
Apply the $\texttt{Poly}(1/2)$ scheme up to round $50^3 \approx 2$ weeks;
\item[-]
Apply the $\texttt{Poly}(1/3)$ scheme up to round $50^4 \approx 2$ years;
\item[-]
Apply the $\texttt{Poly}(1/4)$ scheme up to round $50^5 \approx 100$ years $\ldots$ and so on.
\end{itemize}
Similarly, by setting $\kappa = 5$ rounds ($\approx 1$ minutes in Ethereum),
\begin{itemize}[itemsep = 3 pt]
\item[-]
Apply the $\texttt{Poly}(1)$ scheme up to round $5^2 \approx 4$ minutes;
\item[-]
Apply the the $\texttt{Poly}(1/2)$ scheme up to round $5^3 \approx 20$ minutes; $\cdots$
\item[-]
Apply the $\texttt{Poly}(1/10)$ scheme up to round $5^{11} \approx 15$ years $\ldots$, and so on.
\end{itemize}

\quad It is also possible to tune the parameter $\alpha$ at random time points adaptive to the reward rate. 
That is,
\begin{itemize}[itemsep = 3 pt]
\item
Set $\alpha = \alpha_0 > 0$, and apply the $\texttt{Poly}(\alpha_0)$ scheme up to round $k_0$ 
where $k_0$ is the first time by which no new block is validated in $\kappa$ rounds;
\item
Set $\alpha = \alpha_1 < \alpha_0$, 
and apply the $\texttt{Poly}(\alpha_0)$ scheme up to round $k_1$ 
where $k_1$ is the first time by which no new block is validated in $\kappa$ rounds since then $\ldots$ and so on.
\end{itemize}
Note that in either case the process of stakes is a time-changed P\'olya urn,
so the results in Section \ref{sc3} continue to hold 
(except that the convergence rate will depend on the choice of $\{\alpha_k\}$).



\bibliographystyle{abbrvnat} 
\bibliography{unique.bib} 


\end{document}